\documentclass[11pt]{article}

\usepackage{amssymb}
\usepackage{amsmath}
\usepackage{graphicx}
\usepackage{tikz}
\usepackage{mathrsfs}
\usepackage{float}
\usepackage{makecell}
\usepackage{amscd}
\usepackage{multirow}
\usepackage{cite}
\usepackage{verbatim}
\usepackage{hhline}
\usepackage{comment}
\usepackage{amsthm}
\usepackage{fullpage}

\newcommand{\E}{\mathbb{E}}

\newcommand{\R}{\mathbb{R}}

\newcommand{\kk}{\mathbb{K}}

\newcommand{\Aut}{\text{Aut}}

\renewcommand{\epsilon}{\varepsilon}
\renewcommand{\diamond}{\diamondsuit}
\DeclareMathOperator{\lk}{lk}

\DeclareMathOperator{\fl}{flag}
\usepackage{relsize}
\newtheorem{theorem}{Theorem}[section]
\newtheorem{lemma}[theorem]{Lemma}

\newtheorem{claim}[theorem]{Claim}
\newtheorem{conjecture}[theorem]{Conjecture}
\newtheorem{definition}[theorem]{Definition}
\newtheorem{remark}[theorem]{Remark}
\title{One-sided sharp thresholds for homology of random flag complexes}
\author{Andrew Newman}
\begin{document}
\maketitle
\begin{abstract}
We prove that the random flag complex has a probability regime where the probability of nonvanishing homology is asymptotically bounded away from zero and away from one. Related to this main result we also establish new bounds on a sharp threshold for the fundamental group of a random flag complex to be a free group. In doing so we show that there is an intermediate probability regime in which the random flag complex has fundamental group which is neither free nor has Kazhdan's property (T).
\end{abstract}
\section{Introduction}
One of the most well known and important phase transitions in the Erd\H{o}s--R\'{e}nyi random graph model $G(n, p)$ is the phase transition for the giant component and emergence of cycles at $p = 1/n$ first proved in \cite{ER}. While sharper versions examining the behavior of random graphs in the critical window exist, to motivate the results here we state the Erd\H{o}s--R\'{e}nyi phase transition theorem as follows.
\begin{theorem}[(Erd\H{o}s--R\'{e}nyi \cite{ER})] \label{ERphase}
For $G \sim G(n, c/n)$ with $c \in (0, \infty)$ constant, one has the following:
\begin{enumerate}
\item If $c < 1$ then with high probability $G \sim G(n, c/n)$ has all connected components of order $O(\log n)$ and the probability that $G$ contains cycles is \[1 - \sqrt{1 - c}\exp(c/2 + c^2/4).\]
\item If $c > 1$ then with high probability $G \sim G(n, c/n)$ has a unique giant component of order $\Theta(n)$ and the rest of the components are of order $O(\log n)$ and with high probability $G$ has cycles. 
\end{enumerate}
\end{theorem}
Because there are almost always cycles above the threshold but below threshold there is a positive probability that cycles are present and a positive probability they are not, the cycle threshold in $G(n, p)$ is described as a \emph{one-sided sharp threshold}. Here we study a higher-dimensional analogue of this one-sided sharp threshold for the emergence of cycles in the random clique complex model first introduced by Kahle \cite{KahleRandomClique}. For any graph $G$ the clique complex of $G$, also call the flag complex of $G$, is the simplicial complex $X$ whose faces are the cliques of $G$. In other words the clique complex of $G$ is the maximal simplicial complex $X$ so that the 1-skeleton of $X$, denoted $X^{(1)}$, is $G$. The random clique complex model $X(n, p)$ is sampled as the clique complex of the Erd\H{o}s--R\'{e}nyi random graph $G(n, p)$.

Cycles in $G(n, p)$ generalize to higher dimensions as nonvanishing homology groups in $X(n, p)$. A short argument using classic results  about inclusion of cliques in $G(n, p)$ can be used to show that $X \sim X(n, c/\sqrt[d]{n})$ for $c > \sqrt[d]{d + 1}$ with high probability has nonvanishing $d$th homology group (for any choice of coefficients). The details of such an argument can be found in \cite{KahleRational}. 

On the other hand, a result of \cite{KahleRandomClique} establishes that if $p = n^{-1/d + \epsilon}$ for any fixed $\epsilon > 0$, then with high probability $X \sim X(n, p)$ has vanishing $d$th homology group. As the main result here we improve on this lower bound by showing a probability regime where the probability of that the $d$th homology group vanishes is bounded away from zero and away from one. 

\begin{theorem}\label{maintheorem}
For $d \geq 2$, there exists an explicit constant $\epsilon_d$ so if $c < \epsilon_d$ then $\beta_d(X, \kk)$, for any coefficient field $\kk$, for $X \sim X(n, c/\sqrt[d]{n})$ is asymptotically Poisson distributed with mean 
\[\frac{c^{2(d + 1)d}}{2^{d + 1}(d + 1)!}.\]
\end{theorem}

The Poisson distribution comes from counting particular subcomplexes in $X \sim X(n, c/\sqrt[d]{n})$. The boundary of the $(d + 1)$-dimensional cross-polytope is the clique complex of the complete $(d + 1)$-partite graph with two vertices in each part, and standard results about random graphs tell us that $X \sim X(n, c/\sqrt[d]{n})$ for $c$ constant will contain Poisson distributed copies of the boundary of the $(d + 1)$-dimensional cross-polytope. For $c$ a small enough constant we show that these cross-polytope boundaries are the \emph{only} subcomplexes of $X(n, c/\sqrt[d]{n})$ that can carry homology in degree $d$.

In fact we actually prove something stronger. The proof of Theorem \ref{maintheorem} builds on a result of Malen \cite{Malen} about $d$-collapsibility of random clique complexes. We show that in the studied regime $X \sim X(n, c/\sqrt[d]{n})$ asymptotically almost surely can be collapsed to a complex whose pure $d$-dimensional part is a face-disjoint union of $(d + 1)$-cross-polytope boundaries. We explain this in more detail in Section \ref{sec:background}. 

Because we show a collapsibility result, the proof of Theorem \ref{maintheorem} implies that for $c$ sufficiently small $X \sim X(n, c/\sqrt{n})$ asymptotically almost surely has free fundamental group. This improves on a result of Costa, Farber, and Horak \cite{CostaFarberHorak} that $X(n, n^{-\alpha})$ has free fundamental group if $\alpha > 1/2$, but nonfree fundamental group if $1/3 < \alpha < 1/2$. Here the assumption of $\alpha > 1/3$ comes from the threshold for simply connectivity of $X(n, p)$, due to \cite{CostaFarberHorak} and \cite{KahleRandomClique}. In Section \ref{sec:freegroup} we adapt some of the methods for the proof of Theorem \ref{maintheorem} combined with results of \cite{CostaFarberHorak} to also improve on the upper bound for the threshold for the fundamental group to be nonfree. Our result regarding the fundamental group is the following theorem:

\begin{theorem}\label{pi1theorem}
For $c < \frac{1}{64}$ and $p < c/\sqrt{n}$, asymptotically almost surely $X \sim X(n, p)$ has that $\pi_1(X)$ is free. For $C > \sqrt{3}$, $\alpha > 1/3$, and $C/\sqrt{n} < p < n^{-\alpha}$, asymptotically almost surely $X \sim X(n, p)$ has that $\pi_1(X)$ is not free.
\end{theorem}

To establish the upper bound in Theorem~\ref{pi1theorem} we study ``almost asphericity" of random flag complexes. This is a natural analogue of a notion previously studied in the Linial--Meshulam model by Costa and Farber \cite{CF}. Theorem~\ref{pi1theorem} together with a result of Kahle \cite{KahleRational} implies the following theorem which settles a conjecture of Costa, Farber, and Horak that there is an intermediate regime where $\pi_1(X)$ for $X \sim X(n, p)$ is not free and does not have Kazhdan's property~(T).

\begin{theorem}\label{intermediateregime}
For any fixed $\epsilon > 0$ if
\[\left( \frac{3 + \epsilon}{n} \right)^{1/2} < p < \left( \frac{(3/2 - \epsilon) \log n}{n} \right)^{1/2},\]
then with high probability $X \sim X(n, p)$ has that $\pi_1(X)$ is not free and does not have property~(T).
\end{theorem}

We then close with some discussion about what the true sharp thresholds for the discussed properties might be.

\section{Background}\label{sec:background}
A series of four papers \cite{ALLM, AL2, AL, LP} study the generalization of the Erd\H{o}s--R\'{e}nyi phase transition to the Linial--Meshulam model. Recall that for $d \geq 1$ fixed, one samples $Y \sim Y_d(n, p)$ in the Linial--Meshulam model by starting with the complete $(d - 1)$-complex on $n$ vertices and including each $d$-dimensional face independently with probability $p$; the $d = 1$ case is then exactly the Erd\H{o}s--R\'{e}nyi random graph model. For $d \geq 2$, there are (at least) two ways to generalize the existence of cycles in a graph to an analogous property in a $d$-complex: non-$d$-collapsibility and the nonvanishing of the $d$th homology group. 

A face in a simplicial complex is said to be \emph{free} if it is properly contained in only one other face. The removal of a free face and its unique coface is called an \emph{elementary collapse}, and a complex is $d$-collapsible provided that there exists a sequence of elementary collapses which eliminates all faces of dimension at least $d$. For a graph, a free vertex is simply a leaf and so a graph is acyclic if and only if it is 1-collapsible. On the other hand an elementary collapse is a homotopy equivalence, so a $d$-collapsible complex will have no homology in degree $d$ or larger, however it is well known that the reverse implication holds only for $d = 1$. 

For $Y_d(n, p)$, $d$-collapsibility and nonvanishing of the $d$th homology group have different thresholds, the former is established by \cite{ALLM} and \cite{AL2} and the latter by \cite{AL} and \cite{LP}. For each value of $d$ both thresholds are in the regime $p = c/n$. The critical constants $\gamma_d$ and $c_d$ for each thresholds are given explicitly by solutions to certain transcendental equations with $\gamma_d$ being of order $\log d$ for $d \rightarrow \infty$ and $c_d$ being asymptotically very slightly smaller than $d + 1$. 

Both of these thresholds in $Y_d(n, p)$ are one-sided sharp thresholds. For $p = c/n$, one can show that the expected number of copies of the boundary of the $(d + 1)$-simplex, $\partial \Delta_{d + 1}$, in $Y_d(n, c/n)$ is Poisson distributed with mean $c^{d + 2}/(d + 2)!$. Since $Y_d(n, p)$ is always $d$-dimensional, a $(d + 1)$-simplex boundary is always a nontrivial homology class and an obstruction to $d$-collapsiblity. Thus the results of \cite{ALLM, AL2, AL, LP} can be summarized as follows: 
\begin{itemize}
\item For $0 < c < \gamma_d$, $Y \sim Y_d(n, c/n)$ is $d$-collapsible with probability asymptotic to $\exp(-c^{d + 2}/(d + 2)!)$ \cite{ALLM}.
\item For $\gamma_d < c < c_d$, $Y \sim Y_d(n, c/n)$ is not $d$-collapsible with high probability \cite{AL2}, but the $d$th homology group with real coefficients is generated by embedded copies of $\partial \Delta_{d + 1}$, in particular $\beta_d(Y; \R)$ is asymptotically Poisson distributed with mean $c^{d + 2}/(d + 2)!$ \cite{LP}.
\item For $c_d < c$, $Y \sim Y_d(n, c/n)$ has nonvanishing $d$th homology group for any choice of coefficients with high probablitiy \cite{AL}. 
\end{itemize}

The key to the arguments which establish both thresholds is the local behavior of $Y \sim Y_d(n, p)$ at a $(d - 1)$-dimensional face and the possibility of large cores, complexes which are obstructions to $d$-collapsibility. It seems plausible that the ``local behavior at a $(d - 1)$-dimensional face" part of the arguments for $d$-collapsibility and vanishing of $d$th homology could be adapted to work in the random flag complex setting. The larger challenge though is to adapt an argument counting $d$-dimensional cores in $X \sim X(n, p)$. The Linial--Meshulam model $Y_d(n, p)$ produces complexes that are $d$-dimensional and $d$-complexes are easier to $d$-collapse than complexes of larger dimension. 

A $d$-dimensional complex is said to be a \emph{core} provided that each $(d - 1)$-dimensional face is contained in at least two $d$-dimensional faces, so a $d$-complex is $d$-collapsible if and only if it does not contain a $d$-dimensional core. This means that a greedy approach to elementary collapses will tell us whether or not a given $d$-complex is $d$-collapsible. For $d$-collapsibility of a given $k$-complex with $k > d$ the situation becomes more complicated. For example any simplex is collapsible in the usual sense to a vertex, so in particular it is $d$-collapsible for any $d$, however if the collapses are chosen in the wrong way it's possible to get stuck. An explicit example of this phenomenon is shown \cite{benedetti2009dunce} where Benedetti and Lutz show a way to collapse the 7-simplex to a triangulation of the dunce hat on 8 vertices, which is not 2-collapsible. In Section \ref{sec:overview} we describe how we work around this potential possibility of getting stuck to prove Theorem \ref{maintheorem}. 

A paper of Malen \cite{Malen} studies $d$-collapsibility in $X(n, p)$. His main result shows that for $\alpha > 1/d$, with high probability $X \sim X(n, n^{-\alpha})$ is $d$-collapsible. This result was applied recently by Dochtermann and the present author \cite{DochtermannNewman} to establish results on random coedge ideals. Some of the background for that article motivates the questions considered here, as the $p = n^{1/d}$ regime for $d$ an integer are related to boundary cases for theorems proved in \cite{DochtermannNewman}. In \cite{DochtermannNewman} we consider a conjecture of Erman and Yang \cite{ErmanYang} about normal distribution of Betti numbers of random coedge ideals. Random coedge ideals correspond in a natural way to random flag complexes and in \cite{ErmanYang} Erman and Yang prove a normal distribution of the first row of the Betti table of the random coedge ideal for $G \sim G(n, c/n)$ and $c < 1$, i.e. for $p$ close to, but below, the Erd\H{o}s--R\'{e}nyi phase transition. A challenge to generalizing their conjecture to other rows was that there was not an established one-sided sharp phase transition for homology in dimensions larger than one for $X \sim X(n, p)$. 

As discussed in the introduction, the proof of Theorem \ref{maintheorem} proceeds by proving something stronger related to $d$-collapsibility. Toward formally stating our $d$-collapsibility result, we recall that the pure $d$-dimensional part of a simplicial complex $X$ is the subcomplex of $X$ whose facets are the $d$-dimensional faces of $X$, and we say that a simplicial complex $X$ is \emph{almost $d$-collapsible} if there is a sequence of elementary collapses from $X$ to a complex $Y$ so that $Y$ is at most $d$-dimensional and the pure $d$-dimensional part of $Y$ is a face disjoint union of copies of the $(d + 1)$-dimensional cross-polytope boundary. For convenience, we let $\diamond_d$ denote the $(d + 1)$-dimensional cross-polytope boundary. Because we are dealing with flag complexes throughout, we typically won't distinguish between $\diamond_d$ as a $d$-complex and its 1-skeleton as a graph. The stronger version of Theorem \ref{maintheorem} that we prove is the following.
\begin{theorem}\label{realmaintheorem}
For $c < \frac{1}{2^{2d + 1}d}$ with high probability $X \sim X(n, c/\sqrt[d]{n})$ is almost $d$-collapsible. 
\end{theorem}

If a complex is almost $d$-collapsible then its $d$th homology group is the free abelian group with one generator for each remaining copy of $\diamond_d$ after the collapsing sequence erases everything else of dimension $d$ and higher. So then the Poisson distribution part of Theorem \ref{maintheorem} comes from enumeration of copies of $\diamond_d$ that are not collapsed away by collapsing $X$. This part comes from classic random graph theory.

Recall that the \emph{essential density of a graph $H$} is 
\[\rho(H) = \max\{e(H')/v(H') \mid H' \subseteq H\},\]
and that a graph is strictly balanced provided the essential density is attained by the graph itself and not by any proper subgraphs. A classic result of Bollob\'{a}s \cite{BollobasDensity} shows that for any finite strictly balanced graph $H$, the number of copies of $H$ embedded in $G(n, c/n^{1/\rho(H)})$ for $c$ fixed is asymptotically Poisson distributed with mean:
\[\frac{c^{e(H)}}{|\Aut(H)|}.\]
For more background on embedding subgraphs into random graphs see for example \cite[Chapter~5]{FriezeRandomGraphs}. 

It is straightforward to check that $\diamond_d$ is strictly balanced with essential density $d$. The number of automorphism of $\diamond_d$ is $(d + 1)!2^{d + 1}$; this can be easily seen by counting the number of automorphisms of the complement of $\diamond_d$ which is a matching on $2(d + 1)$ vertices. Therefore the number of copies of $\diamond_d$ embedded in $X(n, c/\sqrt[d]{n})$ is asymptotically distributed as a Poisson random variable with mean 
\[\frac{c^{2d(d + 1)}}{2^{d + 1}(d + 1)!}.\]

We have to be a bit careful though as not every copy of $\diamond_d$ necessarily survives the collapsing process. For example a non-induced copy of $\diamond_d$ could potentially be collapsed away, however in proving Theorem \ref{maintheorem} from Theorem \ref{realmaintheorem} we show that copies of $\diamond_d$ that do not contribute to $d$th homology are negligible in the limit.

\begin{remark}
Theorem \ref{realmaintheorem} and a simple first moment argument enumerating copies of $\diamond_d$ in an Erd\H{o}s--R\'{e}nyi random graph imply that if $p = o(n^{-1/d})$ then $X \sim X(n, p)$ is asymptotically almost surely $d$-collapsible. This strengthens the result of Malen \cite{Malen} that $\alpha > 1/d$ implies that $X \sim X(n, n^{-\alpha})$ is asymptotically almost surely $d$-collapsible and an earlier result of Kahle \cite{KahleRandomClique} that $\alpha > 1/d$ implies $X \sim(n, n^{-\alpha})$ asymptotically almost surely has no homology above dimension $(d - 1)$. 
\end{remark}

\section{Overview of the proof}\label{sec:overview}
The approach to proving Theorem \ref{realmaintheorem} is to count obstructions to $d$-collapsibility. Typically a $d$-dimensional core is defined to be a $d$-dimensional complex so that every $(d - 1)$-dimensional face belongs to at least two $d$-dimensional faces. Here though we drop the ``$d$-dimensional complex" assumption. We say that a simplicial complex $Z$ is a \emph{$d$-core} provided that every $(d - 1)$-dimensional face of $Z$ belongs to at least two $d$-dimensional faces. Note that usage of ``$d$-core" here differs from the definition readers may be familiar with from hypergraph literature.

Naively we might try to show that the only $d$-cores are the embedded copies of $\diamond_d$ and that this implies that the complex collapses so that the only $d$-dimensional faces left are those in copies of $\diamond_d$. This will not work however because, unlike $Y_d(n, c/n)$, $X(n, c/\sqrt[d]{n})$ in general has dimension larger than $d$. In particular for $d \geq 2$ the expected number of $(d + 1)$-dimensional faces in $X(n, c/\sqrt[d]{n})$ is 
\[\binom{n}{d + 2} \left(\frac{c}{\sqrt[d]{n}} \right)^{\binom{d + 2}{2}} \approx \frac{c}{(d + 2)!} n^{(d + 2)(1 - \frac{d + 1}{2d})} \rightarrow \infty.\]
And the boundary of a $(d + 1)$-simplex is a $d$-core.

To describe precisely how we will collapse $X(n, c/\sqrt[d]{n})$ we first introduce some terminology and notation that is fairly standard. For a $d$-complex $Y$, the \emph{dual graph of $Y$}, denoted $\mathcal{G}(Y)$, is defined to be the graph whose vertices are the $d$-dimensional faces of $Y$ with edges between two vertices if and only if the corresponding faces meet at a $(d - 1)$-dimensional face. A \emph{$d$-dimensional strongly connected complex} is a $d$-complex $Y$ so that $\mathcal{G}(Y)$ is connected. We will collapse $X \sim X(n, c/\sqrt[d]{n})$ by collapsing each of it's $d$-dimensional strongly connected components separately, and showing that these collapsing moves may be applied in a coherent way to collapse all the faces of dimension $d$ from $X$. This is similar to the approach taken in \cite{Malen}. The first step is a slight strengthening of Lemma 3.3 of \cite{Malen}, given here with proof as Lemma \ref{partitionlemma}.

In order to state this lemma we introduce a definition and some notation.
\begin{definition}
For a simplicial complex $Y$ the \emph{$k$-flag closure of $Y$} denoted $\fl_k(Y)$ is the maximal simplicial complex whose $k$-skeleton is the same as the $k$-skelelton of $Y$. In otherwords $\fl_k(Y)$ is the simplicial complex obtained from $Y$ by adding the requirement that the simplex $\sigma$ is included if its $k$-skeleton belongs to $Y$. We say that a complex $Y$ is a \emph{$k$-flag complex} if $Y = \fl_k(Y)$. Outside of this section and Section \ref{sec:smallcores} we will most often deal with the usual flag closure $\fl_1(Y)$, so we will use the notation $\overline{Y}$ for $\fl_1(Y)$.
\end{definition}


\begin{lemma}\label{partitionlemma}
Let $X$ be a clique complex and $S_1$ and $S_2$ be strongly connected components of $X^{(d)}$. Then $\dim(\fl_{d - 1}(S_1) \cap \fl_{d - 1}(S_2)) \leq (d - 1)$ and if $\sigma \in \fl_{d - 1}(S_1) \cap \fl_{d - 1}(S_2)$ is $(d - 1)$-dimensional then $\sigma$ is maximal in at least one of $\fl_{d - 1}(S_1)$ and $\fl_{d - 1}(S_2)$.
\end{lemma}
\begin{proof}
Suppose that $\tau \in \fl_{d - 1}(S_1) \cap \fl_{d - 1}(S_2)$ is a $d$-simplex. By the definition of the $(d -1)$-flag closure $\tau^{(d - 1)}$, which is just $\partial \tau$, belongs to $S_1$ and to $S_2$. But the intersection of $S_1$ and $S_2$ is at most $(d - 2)$-dimensional as the strongly connected components of $X^{(d)}$ partition both $d$-faces and nonmaximal $(d - 1)$-dimensional faces. 

Next suppose that $\sigma$ is a $(d - 1)$-dimensional face in $\fl_{d - 1}(S_1)$ and $\fl_{d - 1}(S_2)$ with $\sigma \subseteq \tau_1 \in \fl_{d - 1}(S_1)$ and $\sigma \subseteq \tau_2 \in \fl_{d -1}(S_2)$, $\tau_1, \tau_2$ both $d$-dimensional. Then the $(d-1)$-skeleton of $\tau_1$ belongs to $S_1$ and the $(d - 1)$-skeleton of $\tau_2$ belongs to $S_2$. Moreover as $X$ is a flag complex $\tau_1$ belongs to $S_1$ and $\tau_2$ belongs to $S_2$. But then we can move from $\tau_1$ to $\tau_2$ in $X^{(d)}$ by passing through $\sigma$. This contradicts maximality of $S_1$ and $S_2$. 
\end{proof}

Lemma \ref{partitionlemma} easily implies the following.

\begin{lemma}\label{partitionlemma2}
Let $X$ be a flag complex with $S_1, ..., S_m$ the strongly connected components of $X^{(d)}$. If $\fl_{d - 1}(S_i)$ is almost $d$-collapsible for each $i$ then $X$ is almost $d$-collapsible.
\end{lemma}
\begin{proof}
By Lemma \ref{partitionlemma}, $\fl_{d - 1}(S_1), ..., \fl_{d - 1}(S_m)$ partitions the faces of $X$ of dimension at least $d$. Therefore any collapsing move that does not involve a face of dimension at most $(d -1)$ only affects one of the $\fl_{d - 1}(S_i)$'s. A collapsing move $(\sigma, \tau)$ in $\fl_{d - 1}(S_i)$ collapsing a $(d - 1)$-face $\sigma$ into a $d$-face $\tau$ means that $\sigma$ is not maximal in $\fl_{d - 1}(S_i)$ so by Lemma \ref{partitionlemma}, $\sigma$ is maximal in all the $\fl_{d - 1}(S_j)$ other than $\fl_{d - 1}(S_i)$ that contain it. Since we are only interested in almost $d$-collapsibility, maximal $(d - 1)$-dimensional faces in $\fl_{d - 1}(S_j)$ do not affect the ability to $d$-collapse, so we don't need to use $\sigma$ to almost $d$-collapse any of the other $\fl_{d - 1}(S_j)$'s. Thus we may run a collapsing sequence on $\fl_{d - 1}(S_1)$ that arrives a face disjoint union of copies of $\diamond_d$. Such a collapsing sequence only changes the other $\fl_{d - 1}(S_i)$'s by possibly removing maximal $(d - 1)$-dimensional faces, which leaves all the necessary collapsing moves available to almost $d$-collapse each $\fl_{d - 1}(S_i)$ separately. After separately collapsing each $\fl_{d - 1}(S_1), ..., \fl_{d - 1}(S_m)$ all the $d$-faces remaining belong to copies of $\diamond_d$. Moreover as $\fl_{d - 1}(S_1), ..., \fl_{d - 1}(S_m)$ partition $d$-faces of $X$ and copies of $\diamond_d$ remaining in each individual $\fl_{d - 1}(S_i)$ are face disjoint from one another, the remaining copies $\diamond_d$ after collapsing $X$ are also face disjoint from one another. 
\end{proof}

Having reduced the problem to collapsing each $d$-dimensional strongly connected component separately, we have two lemmas that will constitute the bulk of the work to prove Theorem \ref{realmaintheorem}.

\begin{lemma}\label{firstmomentlargecore}
For $c < \frac{1}{2^{1 + 2d}d}$ with high probability $X \sim X(n, c/\sqrt[d]{n})$ has no strongly connected $d$-dimensional subcomplex on more than $C \log n$ edges for $C$ large enough.
\end{lemma}

\begin{lemma}\label{simplecores}
For $d \geq 2$, if $H$ is a $d$-complex where the link of each $(d - 2)$-dimensional face is an induced subgraph of $H^{(1)}$ and $H^{(1)}$ has essential density
\[\rho(H^{(1)}) < d + \frac{1}{4 + 4d},\]
then $\fl_{d - 1}(H)$ is almost $d$-collapsible.
\end{lemma}

Recall that the link of a face $\sigma$ in a simplicial complex $H$ is the subcomplex of $H$ given by 
\[\lk_H(\sigma) := \{\tau \setminus \sigma \mid \sigma \subseteq \tau \in H\}\]

Of these two lemmas, Lemma \ref{simplecores} is the easier to prove, so we prove it in Section \ref{sec:smallcores} and then prove Lemma \ref{firstmomentlargecore} in Section \ref{sec:largecores}. First we show how these two lemmas imply Theorem \ref{realmaintheorem} and then how Theorem \ref{realmaintheorem} implies Theorem \ref{maintheorem}.
\begin{proof}[Proof of Theorem \ref{realmaintheorem}]
By Lemma \ref{partitionlemma2} it suffices to show that each $d$-dimensional strongly connected component of $X \sim X(n, c/\sqrt[d]{n})$ is almost $d$-collapsible after taking the $(d - 1)$-flag closure. By Lemma \ref{firstmomentlargecore}, with high probability $X \sim X(n, c/\sqrt[d]{n})$ has no $d$-dimensional strongly connected component $S$ so that the graph of $S$ has more than $C \log n$ edges for some $C = C(d, c)$ sufficiently large. So there is no $d$-dimensional strongly connected component on more than $2C \log n$ vertices. 

Next we use Lemma \ref{simplecores} to conclude that with high probability every $d$-dimensional strongly connected component on at most $3C \log n$ vertices has essential density (of its graph) less than $d + \frac{1}{4 + 4d}$. The probability that $X \sim X(n, c/\sqrt[d]{n})$ has a $d$-dimensional strongly connected component on at most $3C \log n$ vertices of essential density at least $d + \frac{1}{4 + 4d}$ is at most the probability that $G \sim G(n, c/\sqrt[d]{n})$ has a subgraph $H$ on at most $3C \log n$ vertices with $|E(H)| \geq (d + \frac{1}{4 + 4d})|V(H)|$.

By a standard first moment argument the expected number of subgraphs of $G(n, c/\sqrt[d]{n})$ with $v$ vertices at least $(d + \epsilon)v$ edges for any $\epsilon > 0$ is at most
\[\binom{n}{v} (v^2)^{(d + \epsilon)v} \left(\frac{c}{\sqrt[d]{n}} \right)^{(d + \epsilon) v}.\]
By taking a sum over $1 \leq v \leq 3C \log n$ we have the probability that $X \sim X(n, c/\sqrt[d]{n})$ contains a strongly connected component on at most $3C \log n$ vertices with density larger than $(d + \epsilon)$ is at most
\begin{eqnarray*}
\sum_{v = 1}^{3C \log n} \binom{n}{v} (v^2)^{(d + \epsilon)v} \left(\frac{c}{\sqrt[d]{n}} \right)^{(d + \epsilon) v} &\leq& \sum_{v = 1}^{3C \log n} \left(n^{\frac{-\epsilon}{d}} v^{2d + 2 \epsilon - 1} e c^{d + \epsilon} \right)^v \\
&\leq& \sum_{v = 1}^{\infty} \left(\frac{ec^{d + \epsilon} (3C \log n)^{2d + 2 \epsilon - 1}}{n^{\frac{\epsilon}{d}}} \right)^v \\
&=& o(1).
\end{eqnarray*}
Thus with high probability each $d$-dimensional strongly connected component of $X \sim X(n, c/\sqrt[d]{n})$ is on at most $2C \log n$ vertices and has essential density less than $d + \frac{1}{4 + 4d}$. Lastly, in order to apply Lemma \ref{simplecores} we check that necessary link condition holds for each strongly connected connected component of $X^{(d)}$; we will see that this holds simply because $X$ is a flag complex rather than for any reason having to do with the randomness. 

As $X$ is a flag complex, if $S$ is a strongly connected component of $X^{(d)}$ then for any $(d - 2)$-face $\sigma$ of $S$ we just have to verify that its link within $S$ is an induced subgraph of $X$. Suppose that $v_1, v_2$ belong to $\lk_S(\sigma)$ and that $\{v_1, v_2\}$ is an edge of $X$. Then $\{v_1\} \cup \sigma$ and $\{v_2\} \cup \sigma$ both belong to $S$ and the 1-skeleton of $\{v_1, v_2\} \cup \sigma$ belongs to $X$. Therefore $\{v_1, v_2\} \cup \sigma$ belongs to $X^{(d)}$ and in particular belongs to some strongly connected component of $X^{(d)}$. As the strongly connected components also partition the non-maximal $(d -1)$-faces, and at least two $(d - 1)$-faces of $\{v_1, v_2 \} \cup \sigma$ belong to $S$, $\{v_1, v_2\} \cup \sigma$ also belongs to $S$ and so $\{v_1, v_2\}$ is an edge of $\lk_S(\sigma)$. 

Thus by Lemma \ref{simplecores} each $d$-dimensional strongly connected component of $X$ is almost $d$-collapsible, so by Lemma \ref{partitionlemma2}, $X$ is almost $d$-collapsible.
\end{proof}
\begin{proof}[Proof of Theorem \ref{maintheorem}]
For $c < \epsilon_d = \frac{1}{2^{2d + 1}d}$, $X \sim X(n, c/\sqrt[d]{n})$ is asymptotically almost surely almost $d$-collapsible. If $X$ is almost $d$-collapsible then after collapsing, the pure $d$-dimensional part of the remaining complex is a face disjoint union of copies of $\diamond_d$. Thus $\beta_d(X; \kk)$ at most the number of copies of $\diamond_d$ in $X$.


Next we give a lower bound that $\beta_d(X; \kk)$ related to the number of copies of $\diamond_d$. Each $\diamond_d$ is a triangulation of $S^d$, so each is a $d$-cycle, so it suffices to bound the probability that each is a $d$-boundary. If $K$ is a set of $2d + 2$ vertices on which there is a copy of $\diamond_d$ in $X$ then a sufficient condition for that copy of $\diamond_d$ to not be a $d$-boundary is that one of the $d$-faces of $\diamond_d$ is maximal in $X$. That is one of the $d$-faces of the flag complex induced on $K$ is not contained in a $(d + 1)$-dimensional face. 

If a copy of $\diamond_d$ contains a $d$-face which is not maximal then either that copy of $\diamond_d$ is not induced or there exists a vertex $v$ outside the $2d + 2$ vertices of $\diamond_d$ that is adjacent exactly to the $(d + 1)$-vertices of a $d$-face of $\diamond_d$. A noninduced copy of $\diamond_d$ is a graph $H$ with $2d + 2$ vertices and at least $2d(d + 1) + 1$ edges, so the expected number of copies of $H$ in $G \sim G(n, c/\sqrt[d]{n})$ is $O(n^{-1/d}) = o(1)$. On the other hand the expected number of copies of $\diamond_d$ that meet a $(d + 1)$-simplex at a $d$-simplex is $O(n^{-1/d})$ too. Thus with high probability 
\[\text{\# copies of $\diamond_d$} - o(1) \leq \beta_d(X; \kk) \leq \text{\# copies of $\diamond_d$}.\]

So asymptotically $\beta_d(X; \kk)$ is equal to the number of copies of $\diamond_d$. By the Bollob\'{a}s result discussed earlier the number of copies of $\diamond_d$ in $X$ is Poisson distributed with mean $c^{2d(d + 1)}/(2^{d + 1} (d + 1)!)$.
\end{proof}

\section{Simplicity of small components}\label{sec:smallcores}
In this section we prove Lemma \ref{simplecores}.  For this proof and elsewhere we need a few more definitions. We have already recalled the definition of a link. We also recall that a complex is \emph{pure} if all of its maximal faces are of the same dimension. For a vertex $v$ in a complex $X$, if $\lk_X(v)$ is $(d - 1)$-collapsible with collapsing sequence $(\sigma_1, \tau_1)$, ..., $(\sigma_t, \tau_t)$ with each $\sigma_i$ a free face of $\tau_i$ and all $\sigma_i, \tau_i \in \lk_X(v)$, to \emph{$d$-collapse around $v$} means to apply the collapsing sequence $(\sigma_1 \cup \{v\}, \tau_1 \cup \{v\})$, ..., $(\sigma_t \cup \{v\}, \tau_t \cup \{v\})$ to $X$. By this collapsing strategy we see that if $\lk_X(v)$ is $(d - 1)$-collapsible and $X \setminus \{v\}$ is $d$-collapsible then $X$ is $d$-collapsible. Indeed we can $d$-collapse $X$ around $v$ and then $v$ is not contained in any faces of dimension larger than $(d - 1)$, but nothing in $X \setminus \{v\}$ has been changed, so we can freely attempt to $d$-collapse $X \setminus \{v\}$ after $d$-collapsing around $v$.

Given the assumptions of Lemma \ref{simplecores}, for brevity we say that a $d$-complex $H$ satisfies the $\star$-condition if the link of each of its $(d - 2)$-faces is an induced subgraph of $H$. We begin with a generalization of the well known fact that $\diamond_d$ is the only non-$d$-collapsible flag complex on fewer than $2d + 3$ vertices. 

\begin{lemma}\label{crosspolytopeminimality}
For $d \geq 2$, if $H$ is a $d$-complex on at most $2d + 2$ vertices which satisfies the $\star$-condition then either
\begin{enumerate}
\item $\fl_{d - 1}(H)$ is $d$-collapsible, or
\item $H = \diamond_d$. 
\end{enumerate}
\end{lemma}

The proof of this lemma will make use of the following for the induction:
\begin{lemma}\label{starinduction}
For $d \geq 3$, if $H$ is a pure $d$-complex that satisfies the $\star$-condition then for any vertex $v$ of $H$, $\lk(v)$ is a pure $(d - 1)$-complex that also satisfies the $\star$-condition. Moreover, $\lk_{\fl_{d - 1}(H)}(v)$ is $\fl_{d - 2}(\lk_H(v))$.
\end{lemma}
\begin{proof}
Obviously since $H$ is pure $d$-dimensional, $\lk_H(v)$ is pure $(d - 1)$-dimensional. We show that it satisfies the $\star$-condition (for faces of codimension $2$, i.e. dimension $d - 3$). Let $\sigma$ be a $(d - 3)$-face of $\lk_H(v)$, and suppose that $\{u_1, u_2\}$ is an edge of $\lk_H(v)$ with $\sigma \cup \{u_1\}$ and $\sigma \cup \{u_2\}$ in $\lk_H(v)$. Then $u_1$ and $u_2$ are both in the link of the $(d - 2)$-face $\sigma \cup \{v\}$, and $\{u_1, u_2\}$ is an edge of $H$ so by the $\star$-condition on $H$, $\sigma \cup \{v, u_1, u_2\}$ is a face of $H$, so $\sigma \cup \{u_1, u_2\}$ belongs to $\lk_H(v)$ and so $\lk_{\lk_H(v)}(\sigma)$ is an induced subgraph of the graph of $\lk_H(v)$. 

Now we verify that $\lk_{\fl_{d - 1}(H)}(v) = \fl_{d - 2}(\lk_H(v))$. Suppose that $\tau \in \lk_{\fl_{d - 1}(H)}(v)$, then $\tau \cup \{v\}$ belongs to $\fl_{d - 1}(H)$ so the $(d - 1)$-skeleton of $\tau \cup \{v\}$ belongs to $H$. Thus the $(d - 2)$-skeleton of $\tau$ belongs to $\lk_H(v)$, so $\tau \in \fl_{d - 2}(\lk_H(v))$. On the other hand suppose that $\tau \in \fl_{d - 2}(\lk_H(v))$. Then the $(d - 2)$-skeleton of $\tau$ belongs to $\lk_H(v)$. Therefore every $(d - 1)$-face of $\tau \cup \{v\}$ that contains $v$ belongs to $H$. Now for any $\sigma \subseteq \tau$ of dimension $(d - 1)$ we have that the $(d - 2)$-skeleton of $\sigma$ belongs to $\lk_H(v)$. Thus it will suffice to show that $\lk_H(v)$ cannot contain an empty $(d - 1)$-simplex. This will imply that $\sigma$ belongs to $\lk_H(v)$ and hence to $H$ which will allow us to conclude that the $(d - 1)$-skeleton of $\tau \cup \{v\}$ belongs to $H$ so $\tau \in \lk_{\fl_{d - 1}(H)}(v)$. 

\begin{claim}
For $d \geq 2$, if $H$ is a pure $d$-complex that satisfies the $\star$-condition then $H$ cannot contain an empty $d$-simplex boundary. 
\end{claim}
\begin{proof}[Proof of claim]
Suppose that $H$ is $d$-dimensional with the $\star$-condition and $H$ contains all the $(d - 1)$-faces of $[v_0, ..., v_d]$. Then $H$ contains $[v_0, ..., v_{d - 2}, v_{d - 1}]$, $[v_0, ..., v_{d - 2}, v_d]$, and $[v_{d - 1}, v_d]$, thus the link of the $(d - 2)$-face $[v_0, ..., v_{d - 2}]$ contains the edge $[v_{d -1}, v_d]$ by the $\star$-condition, and so $[v_0, ..., v_{d - 2}, v_{d - 1}, v_{d}]$ belongs to $H$. 
\end{proof}

By the claim and the fact that $\lk_H(v)$ satisfies the $\star$-condition we have that $\lk_H(v)$ cannot contain an empty $(d - 1)$-simplex, and so we finish the proof. 
\end{proof}

\begin{proof}[Proof of Lemma \ref{crosspolytopeminimality}]
The proof is by induction on $d$. The $d = 2$ case follows from the well known fact that the octahedron $\diamond_2$ is the only flag complex on fewer than 7 vertices which is not 2-collapsible. 
 Now suppose that $d > 2$ and $X$ is a vertex-minimal $d$-complex satisfying the $\star$-condition on at most $2d + 2$ vertices so that it's $(d - 1)$-flag closure is not $d$-collapsible. By minimality we may assume that $X$ is pure $d$-dimensional. If $X$ has a vertex $v$ so that $\deg(v) < 2d$, then $\lk_H(v)$ is a pure $(d - 1)$-complex which also satisfies the $\star$-condition by Lemma \ref{starinduction}. Thus by induction $\lk_H(v)$ must have $(d - 1)$-collapsible $(d - 2)$-flag closure. Therefore by the second part of Lemma \ref{starinduction} taking the $(d - 1)$-flag closure of $H$ allows us to $d$-collapse around $v$ and then use minimality to $d$-collapse $\fl_{d - 1}(H)$. Thus every vertex of $H$ has degree at least $2d$, in particular $H$ has at least $2d + 1$ vertices. If $H$ has exactly $2d + 1$ vertices then the minimum degree condition implies that $H$ has a complete graph as its 1-skeleton. If $H^{(1)}$ is a complete graph then the $\star$-condition implies that every $(d - 2)$-face of $H$ has link whose 1-skeleton is complete. Moreover we can apply the second part of Lemma \ref{starinduction} inductively to conclude that if $\sigma$ is a $k$-face of $H$ then $\lk_{\fl_{d - 1}(H)}(\sigma) = \fl_{d - 1 - (k + 1)}(\lk_H(\sigma))$. Indeed the $k = 0$ case is already part of Lemma \ref{starinduction}. For the inductive step we note that if $\dim(\sigma) = k$ then $\sigma = \tau \cup \{w\}$ for a vertex $w$ and $\tau$ a $(k - 1)$-face with $\lk_H(\tau)$ pure $(d - k)$-dimensional that satisfies the $\star$-condition. Therefore,
\begin{eqnarray*}
\fl_{d - 1 - (k + 1)}(\lk_H(\sigma)) &=& \fl_{d - 1 - (k + 1)}(\lk_{\lk_H(\tau)}(w)) \\
&=& \lk_{\fl_{d - k - 1}(\lk_H(\tau))}(w) \text{ by part (2) of Lemma \ref{starinduction}}\\
&=& \lk_{\lk_{\fl_{d - 1}(H)}(\tau)}(w)  \text{ by induction} \\
&=& \lk_{\fl_{d - 1}(H)}(\sigma)
\end{eqnarray*}

In particular $\lk_{\fl_{d - 1}(H)}(\tau)$ is an ordinary flag complex for $\tau$ a $(d - 3)$-dimensional face. Additionally, by the $\star$-condition every vertex in $\lk_{\fl_{d - 1}(H)}(\tau)$ has that its link is a simplex. Thus $\lk_{\fl_{d - 1}(H)}(\tau)$ is collapsible, and thus we may $d$-collapse around the $(d - 3)$-face $\tau$. At that point $\tau$ does not belong to any $d$-faces of $H$ so we can delete it from $H$ and apply minimality to collapse $H$ the rest of the way. This rules out the possibility that $H^{(1)}$ is a complete graph.

At this point we can apply induction to say that $H$ contains two vertices $w$ and $w'$ so that $\lk_H(w)$ is a copy of $\diamond_{d - 1}$ and $\lk_H(w')$ is a copy of $\diamond_{d - 1}$ on the same vertex set. From here its suffices to show that they must both be the same copy of $\diamond_{d -1}$. This follows in a fairly straightforward way from the $\star$-condition. Suppose that $u$ and $u'$ are antipodal vertices in $\lk_H(w) = \diamond_{d - 1}$ but $\{u, u'\}$ is an edge of $H$. Then taking some $(d - 3)$-face $\tau$ in $\lk_H(w) \cap \lk_H(u) \cap \lk_H(u')$ (which must exist since this intersection is a copy of $\diamond_{d - 2}$) we have $\tau \cup \{w, u\}$, $\tau \cup \{w, u'\}$, and $\{u,u'\}$ all present in $H$ with $\tau \cup \{w\}$ of dimension $d - 2$ so $\tau \cup \{w, u, u'\}$ belongs to $H$ which means that $\{u, u'\}$ is an edge of $\lk(w)$, but this is a contradiction. It therefore follows that the graph of $\lk_H(w)$ and $\lk_H(w')$ are identical and both are graphs of the same flag complex $\diamond_{d - 1}$, so $H$ is $\diamond_d$. 
\end{proof}

\begin{proof}[Proof of Lemma \ref{simplecores}]
Let $H$ be a minimal counterexample, so $H$ is pure $d$-dimensional. If $H$ has a vertex $u$ so that $\lk_{\fl{d - 1}(H)}(u) = \fl_{d - 2}(\lk_H(u))$ is $(d - 1)$-collapsible then we may $d$-collapse $\fl_{d -1}(H)$ around $u$. As this removes all $d$-faces containing $u$ and the graph induced on $H \setminus \{u\}$ also satisfies the density condition and $H \setminus \{u\}$ satisfies the $\star$-condition so its $(d - 1)$-flag closure is almost $d$-collapsible by minimality. Therefore $H$ is almost $d$-collapsible. By Lemma \ref{crosspolytopeminimality}, it follows that $\deg(u) \geq 2d$ must hold for any vertex $u$ of $H$.  Moreover the following claim holds for $H$ too.

\begin{claim}\label{claim:bigneighbor}
Every vertex of $H$ of degree exactly $2d$ has a neighbor of degree larger than $2d$.
\end{claim}
\begin{proof}[Proof of claim]
Let $u$ be a vertex of $H$ of degree exactly $2d$. Then for by minimality and Lemma \ref{crosspolytopeminimality}, $\lk_{\fl_{d - 1}(H)}(v)$ is $\diamond_{d - 1}$ (strictly speaking from what's proved here so far this only applies to $d \geq 3$, but the $d = 2$ case is obvious; the only non 1-collapsible flag complex on 4 vertices is the 4-cycle $\diamond_1$). 

We want to show that some vertex in $\lk_{\fl_{d - 1}(H)}(u)$ must have degree larger than $2d$ in $H$. For contradiction suppose that every vertex in $\lk_{\fl_{d - 1}(H)}(u)$ has degree $2d$ in $H$. Within $\lk_{\fl_{d - 1}(H)}(u)$ each vertex has degree $2d - 2$ since $\lk_{\fl_{d - 1}(H)}(u)$ is the graph of $\diamond_{d - 1}$. Thus each vertex in $\lk_{\fl_{d - 1}(H)}(u)$ can have exactly one neighbor other than $u$ outside $\lk_{\fl_{d - 1}(H)}(u)$ by our assumption that all have total degree $2d$ in $H$.

 Let $w_1$, $w_2$ be two vertices in $\lk_{\fl_{d - 1}(H)}(u)$ so that $w_1$ is adjacent to $u_1$, $w_2$ is adjacent to $u_2$ with $u_1$ and $u_2$ outside of $\lk_{\fl_{d - 1}(H)}(u)$. If $w_1$ and $w_2$ are connected by an edge then the edge $\{w_1, w_2\}$ must belong to a $(d - 1)$-dimensional face $\sigma \in \lk_{\fl_{d - 1}(H)}(u)$, since $\lk_{\fl_{d - 1}(H)}(u) = \diamond_{d - 1}$ is pure $(d - 1)$-dimensional. Now $\sigma$ is contained in at least two $d$-dimensional faces, one of which must be $\sigma \cup \{u\}$. Otherwise if $\sigma$ is only contained in $\sigma \cup \{u\}$ then $(\sigma, \sigma \cup \{u\})$ is a permitted collapsing move which leaves $\lk_{\fl_{d - 1}(H) \setminus \sigma}(u)$ $(d - 1)$-collapsible. So we can $d$-collapse around $u$ and then use minimality to $d$-collapse $(\fl_{d - 1}(H) \setminus \{u\}) \setminus \sigma$ (which is still $(d - 1)$-flag since $\sigma$ is a maximal $(d - 1)$-face in $H \setminus \{u\}$).

Now the only neighbor of $w_1$ outside of $\lk_{\fl_{d -1}(H)}(u)$ is $u_1$, so the only other possible $d$-face that contains $\sigma$ is $\sigma \cup \{u_1\}$. But the same line of reasoning applied to $w_2$ implies that $\sigma \cup \{u_2\}$ is the only possibility for a $d$-face different from $\sigma \cup \{v\}$ to contain $\sigma$. Thus $u_1 = u_2$. 

This means that for every pair of adjacent vertices $w_1$ and $w_2$ their unique neighbor outside of $\lk_{\fl_{d-1}(H)}(u)$ other than $u$ is the same vertex. But since $d \geq 2$, $\lk_{\fl_{d-1}(H)}(u)$ is a connected graph so there is some single vertex $u'$ outside of $\lk_{\fl_{d - 1}(H)}(u)$ with $u \neq u'$ so that $u'$ is adjacent to all neighbors of $u$. Thus we have found a copy of $\diamond_d$ in $H$. This copy of $\diamond_d$ is attached to the complex induced by deleting $u$ and all vertices in $\lk_{\fl_{d - 1}(H)}(u)$, at $u'$. So $H$ is a wedge sum of a smaller $d$-complex ${H'}$ and a copy of $\diamond_d$. Collapsing moves in $\fl_{d - 1}(H')$ don't affect the attached $\diamond_d$, and by minimality $\fl_{d - 1}(H')$ is almost $d$-collapsible, so $\fl_{d - 1}(H)$ is almost $d$-collapsible too. This contradicts our choice of $H$ thus each vertex of degree $2d$ in $H$ is adjacent to a vertex of degree larger than $2d$.
\end{proof}
Let $t \geq 1$ denote the number of vertices of $H$ of degree larger than $2d$. By the claim each vertex of degree $2d$ has a neighbor of degree larger than $2d$ so there are at least $V(H) - t$ edges from the set of vertices of degree $2d$ to the set of vertices of degree larger than $2d$. Thus the degree sum 
\[ \sum_{\{v \in X \mid \deg(v) \geq 2d + 1\}} \deg(v) \geq |V(H)| - t.\]
And 
\[\sum_{\{v \in X \mid \deg(v) = 2d\}} \deg(v) = 2d(|V(H)|  - t).\]
Thus the number of edges of $X$ is at least
\[\frac{(2d + 1)(|V(H)| - t)}{2}. \]

On the other hand the number of edges is also at least
\[\frac{(2d + 1)t + 2d(|V(H)| - t)}{2} = \frac{2d|V(H)| + t}{2}.\]
So 
\[|E(H)| \geq \max \left\{ \frac{2d|V(H)| + t}{2}, \frac{(2d + 1)(|V(H)| - t)}{2} \right\}.\]
These two values are equal when $t = \frac{|V(H)|}{2 + 2d}$ which means that for every value of $t \geq 1$, the number of edges of $H$ is at least
\[\frac{2d|V(H)| + |V(H)|/(2 + 2d)}{2} = \left(d + \frac{1}{4 + 4d}\right) |V(H)|.\]
But this contradicts the density assumption on $H$ as a minimal counterexample.
\end{proof}

\section{Enumeration of large components}\label{sec:largecores}
Here we prove Lemma \ref{firstmomentlargecore}. The argument splits into two parts. We first rule out graphs of large $d$-dimensional strongly connected components that contain $i$-faces of large degree for some $i < (d - 1)$, where degree of an $i$-dimensional face is taken to be the number of $(i + 1)$-dimensional faces that contain it. Then we enumerate the number of possible $d$-dimensional strongly connected complexes which have all $i$-faces having small degree, and apply a first moment argument. 

For the first case, subcomplexes with faces of large degree may be ruled out by examining the global structure of $X \sim X(n, c/\sqrt[d]{n})$. To motivate this let's first consider the case that $d = 2$. In this case the degree of a vertex in $X \sim X(n, c/\sqrt{n})$ is distributed as a binomial with $n - 1$ trials and success probability $\frac{c}{\sqrt{n}}$ and so by Chernoff bound we can conclude that with high probability every vertex has degree at most $(1 + \epsilon) c \sqrt{n}$ for any $\epsilon > 0$. Thus no vertex of $X$ can have any vertex of degree larger than $(1 + \epsilon) c \sqrt{n}$ and so in particular $X$ cannot contain a 2-dimensional strongly connected subcomplex whose graph has a vertex of degree larger than $(1 + \epsilon) c \sqrt{n}$. 

Now we make this precise and prove the following for general $d \geq 2$.
\begin{lemma}\label{largedegreecores}
For $X \sim X(n, c/\sqrt[d]{n})$ with high probability for every $1 \leq i < d$ and every $\epsilon > 0$ the maximum degree of an $(i-1)$-dimensional face is at most $(1 + \epsilon) c^i n^{1 - i/d}$.
\end{lemma}
\begin{proof}
Given an $(i-1)$-dimensional face $\sigma$ the degree of $\sigma$ in $X$ conditioned on $\sigma$ being included in $X$ in the first place is distributed as a binomial random variable with $n - i$ trials and success probability $\left(\frac{c}{\sqrt[d]{n}}\right)^i$. Therefore the expected degree of $\sigma$ given that $\sigma$ is included in $X$ is asymptotically $c^i n^{1 - i/d}$. Thus by Chernoff bound the probability that $\deg(\sigma) > (1 + \epsilon) c_i n^{1 - i/d}$ is at most $\exp(-C n^{1 - i/d})$ where $C$ is a constant that depends on $c$ and $\epsilon$. Thus taking a union bound over all $\binom{n}{i}$ possible $(i - 1)$-dimensional faces we have that with high probability there is no $i$-dimensional face of degree larger than $(1 + \epsilon) c^i n^{1 - i/d}$, and then we take a further union bound over all $1 \leq i < d$.
\end{proof}
It follows immediately from Lemma \ref{largedegreecores} that for any $\epsilon > 0$, $X \sim X(n, c/\sqrt[d]{n})$ has no strongly connected $d$-dimensional subcomplexes for which there exists $1 \leq i < d$ so that some $(i - 1)$-dimensional face of the subcomplex has degree larger than $(1 + \epsilon) c^i n^{1 - i/d}$. We will say that $Y$ a subcomplex of $X$ is \emph{$c$-bounded} if the degree of every $(i-1)$-dimensional face for $1 \leq i < d$ is at most $c^i n^{1 - i/d}$. We next count the number of $c$-bounded strongly connected $d$-complexes $Y$ with $Y^{(1)}$ having $r$ edges.

\begin{lemma}\label{largecorelemma}
For any $0< c < 1$ and $d \geq 2$, the number of graphs $H$ on $[n]$ with $r$ edges so that $H$ is the graph of a $c$-bounded strongly connected $d$-dimensional complex is at most
\[2^{d + 1} n^{(d + 1)/2} (2^{1 + 2d} d \sqrt[d]{n})^r.\]
\end{lemma}
The approach we take to prove this lemma is inspired by the proof of Theorem 4.1 of \cite{ALLM} where the Linial--Meshulam model is considered rather than the random flag complex model. 
\begin{proof}
We consider a procedure to construct a minimal $d$-dimensional strongly connected complex with $r$ edges and at most $n$ vertices by adding $d$-faces one at a time. As the complex we end up with is strongly connected we may construct it by starting with a single $d$-face and at each step adding a $d$-face which contains an already existing $(d - 1)$-face which itself belongs to an existing $d$-face. We will always add our $d$-faces in this way, but this isn't quite enough to get the bound we want because even if we add $d$-faces according to this rule there will be in general many different orderings of $d$-faces which produce the same complex. Additionally, we know that we want $r$ edges at the end, but that doesn't tell us precisely how many steps we need since the number of steps is determined by the number of $d$-faces.
To limit the amount of overcounting, we add a few more requirements to the procedure. First we order the $(d - 1)$-faces of the complete $(d - 1)$-complex on $[n]$. Throughout the process $(d - 1)$-faces of the complex we have constructed so far will be marked as \emph{saturated} or \emph{unsaturated}. At each step the ordering and saturated/unsaturated markings tell us a well defined $(d - 1)$-face $\sigma$ and we \emph{grow from $\sigma$} by adding some new $d$-face $\sigma \cup \{v\}$ to our complex.

If a $(d - 1)$-face $\sigma$ is unsaturated that means it is eligible to be selected later in the process when we will add $\sigma \cup \{v\}$ for some vertex $v$ so that some edge of $\sigma \cup \{v\}$ is missing right we we add $\sigma \cup \{v\}$ to $\sigma$. Saturated means that $\sigma$ is ineligible to grow from for the rest of the process, though new $d$-dimensional cofaces of $\sigma$ may still be added in moves that add $\sigma \cup \{v\}$ but grow from some other $(d - 1)$ face or moves that grow from $\sigma$, but simply fill in an existing $d$-simplex boundary. 

At the first step we simply take an arbitrary $d$-face and mark each of its $(d - 1)$-faces as unsaturated or saturated. At each subsequent step we pick the minimum unsaturated $(d-1)$-face $\sigma$ according to the ordering on edges and add a new $d$-face $\sigma \cup \{y\}$ where $y$ is a vertex, unless there is an empty $d$-simplex boundary in the complex (which we'll discuss later). Of course we also add any missing edges of $\sigma \cup \{y\}$ as well. Other than the first move which adds one $d$-face and $\binom{d + 1}{2}$ edges this procedure partitions the remaining moves into $d + 1$ types: moves which add no edges (i.e. moves that simply fill in an empty $d$-simplex whose 1-skeleton is already present), moves that add a 1 edge, moves that add 2 edges, ..., and moves that add $d$ edges. If we denote the number of these moves in a single instance of the procedure respectively as $m_0, m_1, m_2, ..., m_d$ we observe that the number of edges in the complex we end up with is $\binom{d + 1}{2} + m_1 + 2m_2 + \cdots + dm_d$, and the number of $d$-faces is $1 + m_0 + m_1 + m_2 + \cdots + m_d =: m$. For a complex $K$ constructed by this procedure we let $\Xi_i(K)$ for $i \in \{0, ..., d\}$ denote the number of ways to extend $K$ by adding a single $d$-face in a way that adds exactly $i$ edges and by $\Xi(K)$ the sum $\Xi_0(K) + \Xi_1(K) + \Xi_2(K) + \cdots + \Xi_d(X)$. 

Except for moves where we add no new edges, after we have grown from the minimal $(d - 1)$-face $\sigma$ to $\sigma \cup \{y\}$ we decide for each unsaturated edge of $\sigma \cup \{y\}$ whether or not to flip it to saturated. We additionally need one more rule for the procedure to avoid too large of an  overcount and use the fact that we build flag complexes. We order the $\binom{n}{d+1}$ $d$-faces on $[n]$ and whenever we have a complex $K$ which has an empty $d$-simplex we choose the smallest empty $d$-simplex in $K$ according to the ordering and fill it in growing for its minimal $(d - 1)$-face, and we do nothing with flipping unsaturated $(d - 1)$-faces to saturated. So we can effectively skip over these moves. 

Subject to these rules we choose a sequence $(a_1, ...., a_{t - 1})$ were each $a_i$ is 1, 2, ..., or $d$ with $m_k$ denoting the number of $i$ so that $a_i = k$ so that $\binom{d + 1}{2} + m_1 + 2m_2 + \cdots + dm_d = r$. We'll count the number of ways to run this procedure by steps $K_1 \subseteq K_2 \subseteq \cdots \subseteq K_t$ where $K_1$ is an arbitrary $d$-face and $K_{i + 1}$ is obtained from $K_i$ by a move which adds $a_i$ edges and afterwards filling in all $d$-faces whose 1-skeleton is present, and $t$ is such that $K_t$ will have $r$ edges.

For a given sequence $(a_1, ..., a_{t - 1})$ we let $K_1 \subseteq K_2 \subseteq \cdots \subseteq K_t$ be a run of our procedure from a single triangle $K_1$ to a complex with $r$ edges $K_r$ where add $a_i$ edges to go from $K_i$ to $K_{i + 1}$. We wish to bound 
\[\prod_{i = 1}^{t - 1} \Xi_{a_i}(K_i).\]
At each step the $(d - 1)$-face we grow from is deterministic; it is the minimal unsaturated $(d - 1)$-face in $K_i$. By the $c$-bounded condition if $a_i = k < d$ then the number of ways to extend the selected $(d - 1)$-face at $K_i$ is at most $2^{d + 1}\binom{d}{k} c^{d - k} n^{1 - (d - k)/d} \leq 2^{d + 1} 2^d   c^{d - k} n^{k/d}$. Indeed in order to add only $k$ new edges we have to pick $v$ to be a common neighbor of some $d - k$ vertices of $\sigma$ and add $\sigma \cup \{v\}$. If $a_i = d$ then there are trivially at most $n$ ways to extend $K_i$. After adding the new $d$-face, we choose for each of its $(d + 1)$ facets whether to flip from unsaturated to saturated.

Thus having selected $(a_1, ..., a_{m - 1})$, 
\begin{eqnarray*}
\prod_{i = 1}^{r - 1} \Xi_{a_i}(K_i) &\leq& \left( \prod_{k = 1}^{d - 1} (2^{d + 1} 2^d c^{d - k} n^{k/d})^{m_k} \right)(2^{d + 1}n)^{m_d} \\
&\leq& (2^{1 + 2d})^r c^{d(m_1 + \cdots + m_{d-1}) - m_1 - 2m_2 - \cdots - (d - 1)m_d} \left(\sqrt[d]{n}\right)^{m_1 + 2m_2 + \cdots + dm_d}.
\end{eqnarray*}
As $m_1 + 2m_2 + \cdots + dm_d = r - \binom{d + 1}{2}$ we have that 
\begin{eqnarray*}
&&c^{d(m_1 + \cdots + m_{d-1}) - m_1 - 2m_2 - \cdots - (d - 1)m_d} \left(\sqrt[d]{n}\right)^{m_1 + 2m_2 + \cdots + dm_d} \left(c/\sqrt[d]{n} \right)^{r - \binom{d + 1}{2}} \\
&& =  c^{d(m_1 + \cdots + m_{d - 1}) - (r - \binom{d + 1}{2} - dm_d) + r - \binom{d + 1}{2}} \\
&&\leq c^{r - \binom{d + 1}{2}}. \\
\end{eqnarray*}
So 
\[(2^{1 + 2d})^r c^{d(m_1 + \cdots + m_{d-1}) - m_1 - 2m_2 - \cdots - (d - 1)m_d} \left(\sqrt[d]{n}\right)^{m_1 + 2m_2 + \cdots + dm_d} \leq (2^{1 + 2d})^r(\sqrt[d]{n})^{r - \binom{d + 1}{2}}.\]

Recall though that this is after choosing the sequence $(a_1, ..., a_{m - 1})$ and the initial $d$-dimensional face and saturated/unsaturated markings on its $(d-1)$-face, so the number $c$-bounded strongly connected $d$-complexes on $r$ edges with $[n]$ is at most
\[(2n)^{d + 1} d^r (2^{1 + 2d})^r (\sqrt[d]{n})^{r - \binom{d + 1}{2}} \leq 2^{d + 1} n^{(d + 1)/2} (2^{1 + 2d} d \sqrt[d]{n})^r.\]

\end{proof}

Having proved Lemmas \ref{largedegreecores} and \ref{largecorelemma} we can prove Lemma \ref{firstmomentlargecore}.

\begin{proof}[Proof of Lemma \ref{firstmomentlargecore}]
Take $\epsilon$ so that $c + \epsilon < \frac{1}{2^{1 + 2d}d}$. By Lemma \ref{largedegreecores}, $X \sim X(n, c/\sqrt[d]{n})$ does not contain a strongly connected $d$-dimensional subcomplex (of any size) which is not $(c + \epsilon)$-balanced. By Lemma \ref{largecorelemma} for any $r$ the expected number of $(c + \epsilon)$-bounded strongly connected $d$-complexes in $X$ on $r$ edges is at most
\[2^{d + 1} n^{(d + 1)/2} (2^{1 + 2d}d \sqrt[d]{n})^r \left(\frac{c}{\sqrt[d]{n}} \right)^r = 2^{d + 1} n^{(d + 1)/2} (2^{1 + 2d} dc)^r.\]
Thus by Markov's inequality the probability that there exists $r \geq C \log n$ so that $X \sim X(n, c/\sqrt[d]{n})$ has a $(c + \epsilon)$-bounded strongly connected subcomplex on $r$ edges is at most
\begin{eqnarray*}
2^{d + 1} \sum_{r = C \log n}^{\infty} n^{(d + 1)/2} (2^{1 + 2d}dc)^r &\leq& 2^{d + 1} n^{(d + 1)/2} (2^{1 + 2d}dc)^{C \log n} \frac{1}{1 - 2^{1 + 2d}dc} \\
&=& O(n^{(d + 1)/2 - C \log((2^{1 + 2d}dc)^{-1})}).
\end{eqnarray*}
And this is $o(1)$ for $C$ large enough since $2^{1 + 2d}dc < 1$.
\end{proof}
With this we have now completed the proof of Theorem \ref{maintheorem}

\section{Freeness of the fundamental group}\label{sec:freegroup}
A paper of Costa, Farber, and Horak \cite{CostaFarberHorak} establishes coarse thresholds for the cohomological dimension of the fundamental group of a random clique complex. Namely they show that if $p = n^{-\alpha}$ then the cohomological dimension of $\pi_1(X)$ for $X \sim X(n, p)$ asymptotically almost surely is:
\begin{itemize}
\item 1 if $\alpha > 1/2$
\item 2 if $1/2 >  \alpha  > 11/30$
\item $\infty$ if $11/30 > \alpha > 1/3$.
\end{itemize}
Combined with an earlier result of Kahle \cite{KahleRandomClique} that $1/3 > \alpha$ implies $X$ is simply connected with high probability, this gives the full behavior of the fundamental group of a random flag complex up to the right exponent for the phase transitions. Using Theorem \ref{realmaintheorem}, we can immediately improve on the lower bound for freeness of $\pi_1(X)$. Using existing results in the literature and a new result similar to Lemma \ref{simplecores} we can also improve on the upper bound, proving Theorem \ref{pi1theorem}.\\

The subcritical case of Theorem \ref{pi1theorem} follows immediately from Theorem \ref{realmaintheorem}, via the following lemma which is easy to prove.
\begin{lemma}\label{almost2collapsible}
If $X$ is an almost 2-collapsible flag complex then $\pi_1(X)$ is a free group.
\end{lemma}
\begin{proof}
Since $X$ is almost 2-collapsible, $X$ is homotopy equivalent to a complex of dimension at most 2 whose pure $2$-dimensional part consists of face disjoint copies of $\diamond_2$. If we remove a triangle from each of these copies of $\diamond_2$ we don't change the fundamental group, because the boundary of the triangle we remove is nullhomotopic. After removing a triangle from each copy of $\diamond_2$ we may collapse the complex to a graph.
\end{proof}

The supercritical case of Theorem \ref{pi1theorem} is much more involved, but the necessary lemmas are mostly already in place. To state what we need precisely we say that a (not necessarily flag) complex $X$ is \emph{almost aspherical} if any complex $\widetilde{X}$ obtained from $X$ by removing one and only one triangle from each embedded copy of $\diamond_2$ is aspherical i.e. $\pi_j(\widetilde{X}) = 0$ for all $j \geq 2$. This is similar to the notion of \emph{asphericable} in \cite{CF} which establishes a result about the Linial--Meshulam model being aspherical for a particular range of $p$ once a triangle is removed from each tetrahedron boundary. We use the term ``almost aspherical" rather than ``asphericable" here only for the analogy to ``almost $d$-collapsible". Toward the supercritical part of Theorem \ref{pi1theorem} we prove the following.

\begin{theorem}\label{almostaspherical}
If $p < n^{-\alpha}$ for $\alpha > 12/25$ then $X \sim X(n, p)$ is almost aspherical with high probability. Moreover the cohomological dimension of $\pi_1(X)$ is at most 2. 
\end{theorem}
This theorem together with other results in the literature then implies the following theorem which is exactly the supercritical case of Theorem \ref{pi1theorem}.
\begin{theorem}\label{supercriticalpi1}
For $c > \sqrt{3}$ and $\alpha > 1/3$, if $\frac{c}{\sqrt{n}} < p < n^{-\alpha}$ then $X \sim X(n, p)$ asymptotically almost surely has that $\pi_1(X)$ is not a free group.
\end{theorem}
For this proof and the proof of Theorem \ref{intermediateregime} we recall the following result of Kahle.
\begin{theorem}\cite[Theorem 1.2]{KahleRational}\label{propertyTtheorem}
Let $\epsilon > 0$ be fixed and $X \sim X(n, p)$. If 
\[p \geq \left(\frac{(3/2 + \epsilon) \log n}{n} \right)^{1/2}\]
then with high probability $\pi_1(X)$ has property~(T), and if

\[\frac{1 + \epsilon}{n} \leq p \leq \left( \frac{(3/2 - \epsilon) \log n}{n} \right)^{1/2}\]
then with high probability $\pi_1(X)$ does not have property~(T).

\end{theorem}
\begin{proof}[Proof of Theorem \ref{supercriticalpi1}]
Set $\alpha' \in (12/25, 1/2)$ if $c/\sqrt{n} < p <  n^{-\alpha'}$  we have that $X \sim X(n, c/\sqrt{n})$ is almost aspherical and has cohomological dimension at most 2 with high probability by Theorem \ref{almostaspherical}. To show the cohomological dimension is at least 2, it suffices to prove that $X$ with a single triangle removed from each $\diamond_2$ has homology in degree 2 with high probability. If this holds then after removing a face from every copy of $\diamond_2$ the resulting complex $\widetilde{X}$ has $\pi_1(\widetilde{X}) = \pi_1(X)$ with $\widetilde{X}$ a $K(\pi_1(X), 1)$. By uniquess of the homotopy type of a $K(G, 1)$, if $\pi_1(X)$ is free then $\widetilde{X}$ is homotopy equivalent to a bouquet of circles, but this cannot the case if $\widetilde{X}$ has homology in dimension 2.

By the first moment method with high probability $X$ contains at most $O(n^{6/25} \log n)$ copies of $\diamond_2$. On the other hand the expected number of triangles in $X \sim X(n, p)$ is
\[ \E(f_2(X)) = \binom{n}{3} p^3, \]
 the expected number of edges in $X \sim X(n, c/\sqrt{n})$ is
\[ \E(f_1(X)) = \binom{n}{2} p,\]
and the expected number of tetrahedra in $X \sim X(n, c/\sqrt{n})$ is
\[ \E(f_3(X)) = \binom{n}{4} p^6.\] 
Now 
\[\beta_2(X) \geq f_2(X) - f_1(X) - f_3(X).\]
Moreover $f_i(X)$ simply counts $(i + 1)$-cliques in $G \sim G(n, p)$, and the number of $(i + 1)$-cliques in $G(n, p)$ is known to be well concentrated around its mean when $p$ is such that the mean tends to infinity.
Therefore with high probablilty 
\[\beta_2(X) \geq (1 - o(1))\left(\frac{c^3}{6} - \frac{c}{2}\right) n^{3/2} \geq \delta n^{3/2},\]
for some constant $\delta > 0$ since $c^2 > 3$.
Thus even after deleting a face from each copy of $\diamond_d$, the remaining complex still has lots of homology in degree 2. Thus the cohomological dimension of $\pi_1(X)$ is at least 2. An argument of this type showing that $c > \sqrt[d]{d + 1}$ implies that $X \sim X(n, c/\sqrt[d]{n})$ has homology in degree $d$ also appears in \cite{KahleRational}.

For $n^{-\alpha'} \leq p < n^{-\alpha}$, where $\alpha' < 1/2$ we have $p > \left(\frac{2 \log n}{n} \right)^{1/2}$ which implies by Theorem \ref{propertyTtheorem} that $\pi_1(X)$ has property~(T). A nontrivial group with property~(T) cannot be a free group because property~(T) implies finite abelianization. Moreover since $p < n^{-\alpha}$ we know by the result of \cite{CostaFarberHorak} discussed above that $\pi_1(X)$ is nontrivial.
\end{proof}

\begin{proof}[Proof of Theorem \ref{intermediateregime}]
Theorem \ref{supercriticalpi1} and Theorem \ref{propertyTtheorem} immediately imply Theorem \ref{intermediateregime}.
\end{proof}

It remains to prove Theorem \ref{almostaspherical}. Key to the proof is Theorem 6.1 of \cite{CostaFarberHorak}.

\begin{lemma}\cite[Theorem 6.1]{CostaFarberHorak}\label{CFH61}
Assume that $p = o(n^{-\alpha})$ for $\alpha > 1/3$ fixed. Then for $X \sim X(n, p)$ and $L = L(\alpha)$ a large enough constant the 2-skeleton of $X^{(2)}$ has the following property with high probability: Any subcomplex $Y$ of $X^{(2)}$ is aspherical if and only if every subcomplex $Z \subseteq Y$ having at most $L$ edges is aspherical.
\end{lemma}

The proof of Theorem \ref{almostaspherical} requires showing that $X \sim X(n, p)$ satifies several more properties in addition to the necessary and sufficient condition property of Lemma \ref{CFH61}.

\begin{lemma}\label{additionalproperties}
For $ X \sim X(n, p)$ with $p < n^{-\alpha}$, $\alpha > 12/25$ with high probability $X$ satisfies all of the following properties:
\begin{enumerate}
\item $\dim(X) \leq 4$.
\item Every embedded copy of $\diamond_2$ in $X$ has all of its triangles maximal. In particular every embedded copy of $\diamond_d$ is induced.
\item $X$ does not contain a tetrahedron and a 4-simplex that meet at a triangle.
\item $X$ is 3-collapsible.
\item There exists some constant $L$ so that any subcomplex $Y$ of $X^{(2)}$ is aspherical if and only if every subcomplex $Z \subseteq Y$ having at most $L$ vertices is aspherical.
\item Every for every subcomplex $Z$ on at most $\ell$ vertices for any constant $\ell$, $\rho(Z^{(1)}) < 25/12$.
\end{enumerate}
\end{lemma}

In establishing that $11/30 < \alpha < 1/2$ $X \sim X(n, n^{-\alpha})$ has fundamental group of cohomological dimension 2 in \cite{CostaFarberHorak}, Costa, Farber, and Horak show that there is a subcomplex $Y \subseteq X^{(2)}$ so that every bounded subcomplex of $Y$ collapses to a graph and so that $\pi_1(Y) = \pi_1(X)$. Our approach will be to show that for $\alpha > 12/25$, $p < n^{-\alpha}$, the complex $Y \subseteq X^{(2)}$ obtained from $X \sim X(n, p)$ by collapsing away everything of dimension larger than 2 from $\widetilde{X}$ has $\pi_1(Y) = \pi_1(\widetilde{X}) =  \pi_1(X)$ and every finite subcomplex of $Y$ is aspherical. While it may be possible to show that every bounded subcomplex of this complex $Y$ collapses to a graph by adapting the approach in \cite{CostaFarberHorak}, since we only need that every bounded subcomplex is aspherical we take a different approach based on simple homotopy equivalence.

Simple homotopy equivalence was first described by Whitehead, and is based on elementary collapses and expansions. We've already defined an elementary collapse, and an elementary expansion is simply the reverse of an elementary collapse. That is, if $X$ is a simplicial complex and $\sigma$ is a simplex on a subset of the vertices of $X$ so that exactly one facet $\tau$ of $\partial \sigma$ does not belong to $X$ then the \emph{elementary expansion of $X$ at $\sigma$} is adding $\tau$ and $\sigma$ to $X$. Observe that an elementary expansion is a homotopy equivalence. Two simplicial complexes are \emph{simple homotopy equivalent} provided that there is a sequence of collapses and expansions that transform one into the other.

To prove Theorem \ref{almostaspherical} we show that if $X$ is a flag complex that satisfies the conditions of Lemma \ref{additionalproperties} we can remove a triangle from each copy of $\diamond_2$ in $X$ and then collapse to a 2-complex $Y$ with the property that any $Z \subseteq Y$ on at most $L$ vertices is simple homotopy equivalent to a graph, and is therefore aspherical. In order to prove this we first establish the following deterministic result that is similar to Lemma \ref{simplecores}.

\begin{lemma}\label{Zlemma}
If $Z$ is a simplicial complex with the property that the only triangulated 2-spheres contained in $Z^{(2)}$ are the boundaries of tetrahedra belonging to $Z$ and if the essential density of the graph of $Z$ satisfies
\[\rho(Z^{(1)}) < 25/12,\]
then $Z$ is simple homotopy equivalent to a graph. 
\end{lemma}

For brevity we say that a simplical complex whose only embedded 2-spheres are boundaries of included tetrahedra is \emph{essentially 2-sphere free}. The fact that $Z$ is essentially 2-sphere free will imply that, even though $Z$ may not be a flag complex there is a sequence of expansions that transform any the link of any vertex $v$ to a flag complex while maintaining the essentially 2-sphere free property and density property on $Z \setminus \{v\}$. So if $Z$ has a vertex $v$ with $\overline{\lk_Z(v)}$ 1-collapsible, then $Z$ cannot be a minimal counterexample. From there we can basically follow the argument of the proof of Lemma \ref{simplecores} and afterwards immediately prove Theorem \ref{almostaspherical}.

\begin{proof}[Proof of Lemma \ref{Zlemma}]
Suppose that $Z$ is essentially 2-sphere free with the density condition. Two immediate observations for the density condition are that $Z$ does not contain the 1-skeleton of a 2-sphere on more that 6 vertices and that $\dim(\overline{Z}) \leq 4$. By Euler characteristic a triangulated 2-sphere on $v$ vertices has density $3 - 6/v$, and $3 - 6v \geq 25/12$ for $v \geq 7$. The dimension of $\overline{Z}$ is at most 4 because $Z^{(1)}$ cannot contain a 6-clique as a 6-clique has density $15/6 > 25/12$.

Let $v$ be a vertex of $Z$, we claim that there exists a sequence of expansions each adding a pair of faces $(\sigma, \sigma \cup \{v\})$, so that the resulting complex $Z'$ has $\lk_{Z'}(v)$ is the flag closure $\overline{\lk_Z(v)}$ and $Z' \setminus \{v\}$ still satisfies the essentially 2-sphere free and density assumptions of $Z$. 

Since $\dim(\overline{Z}) \leq 4$, $\dim(\overline{\lk_Z(v)}) \leq 3$, thus if we can show that there is a sequence of expansions first filling in all the triangles of $\lk_Z(v)$ and then from there filling in all empty tetrahedra, that will be sufficient for $Z'$ so that $\lk_{Z'}(v) = \overline{\lk_Z(v)}$, and all that will remain is to check that $Z' \setminus \{v\}$ still satisfies the assumptions from the statement. Suppose that $\partial \sigma$ is an empty triangle in $\lk_Z(v)$. If $\sigma \notin Z$ then $(\sigma, \sigma \cup \{v\})$ is a valid expansion filling in $\partial \sigma$ in the link. However if $\partial \sigma$ is an empty triangle in $\lk_Z(v)$ and $\sigma \in Z$ then we have the boundary of $\sigma \cup \{v\}$ in $Z$ without $\sigma \cup \{v\}$ in $Z$. But this is a triangulated 2-sphere that isn't the boundary of a tetrahedron in $Z$, so this breaks the essentially 2-sphere free condition. Thus $\sigma \notin Z$. Therefore it is possible to fill in the triangles of the link of $v$. 

Once the triangles are filled in we move on to tetrahedra. Empty tetrahedera in the link of $v$ after expanding to fill in all the triangles split into two types. On one hand we have empty tetrahedra that contain at least one triangle added by an expansion and on the other we have empty tetrahedra initially present in $\lk_Z(v)$. For the former case suppose that $\partial \tau$ is an empty tetrahedron in the link of $v$ with $\sigma$ a triangle of $\tau$ added by an expansion. Since $\sigma$ was added by an expansion, the only tetrahedron that contains $\sigma$ is $\sigma \cup \{v\}$, so in particular $\tau$ is not in the complex, so $(\tau, \tau \cup \{v\})$ is an allowed expansion that fills in $\partial \tau$ in the link. If $\partial \tau$ was already in $\lk_Z(v)$ as an empty tetrahedron, then $Z$ being essentially 2-sphere free implies that $\tau$ is not in $Z$. Indeed if $\tau$ belongs to $Z$ and $\partial \tau \in \lk_Z(v)$, then $\partial (\tau \cup \{v\})$ belongs to $Z$, but the boundary of a 4-simplex contains a 2-sphere on 5 vertices in its 2-skeleton. It follows that $\tau \notin Z$ and so $(\tau, \tau \cup \{v\})$ is a permitted expansion. Thus it is possible to expand $Z$ to $Z'$ where $\lk_{Z'}(v) = \overline{\lk_Z(v)}$. 

We claim next that $Z' \setminus \{v\}$ is essentially 2-sphere free. As the expansion moves from $Z$ to $Z'$ never added any edges, the graph of $Z'$ is a subgraph of the graph of $Z$. Thus $Z'$ cannot contain a sphere on more than 6 vertices. If $S \subseteq Z' \setminus \{v\}$ is a triangulated 2-sphere on 6 vertices then at least one triangle $\sigma$ of $S$ belongs to $\lk_{Z'}(v)$. In this case then $\partial \sigma \in \lk_{Z}(v)$, but then replacing $\sigma$ in $S$ by the cone over its boundary with cone point $v$ finds the 1-skeleton of a sphere on 7 vertices in $Z$, but we've already ruled that out. 

Next suppose that $S \subseteq Z' \setminus \{v\}$ is a sphere on 5 vertices. Then some nonempty subset $\{\sigma_1, ..., \sigma_t\}$ of the triangles of $S$ belong to $\lk_{Z'}(v)$. Thus all the edges of the complex $Y$ generated by $\sigma_1, ..., \sigma_t$ belong to $\lk_Z(v)$. If $Y$ contains a vertex of degree at least 4 then $Z$ contains the graph of a 2-sphere $S$ on 5 vertices with an additional vertex adjacent to all vertices of $S$. Such a graph has 14 edges and 6 vertices and $14/6 > 25/12$ breaking the density condition. Thus $Y$ must be a subcomplex of $S$ in which all vertices have degree at most 3. But then $Y$ is either a single triangle, two triangles meeting at an edge, or a tetrahedron with a face removed. In any of these case $Y^{(1)}$ belongs to $\lk_Z(v)$ and $Y$ is a disk. Thus by replacing $Y$ with the cone over its boundary with cone point $v$ we find a copy of a sphere on 5 or 6 vertices in $Z$. But $Z$ is essentially 2-sphere free.

Lastly suppose that $Z'$ contains an empty tetrahedron $S$. Again at least one triangle of $S$ belongs to $\lk_{Z'}(v)$. Let $Y$ be the subcomplex of $S$ generated by the triangles of $S$ in $\lk_{Z'}(v)$. If $Y$ is on at most 2 triangles, the $Y$ is a disk with no internal vertices and the boundary of that disk belongs to $\lk_Z(v)$, so we can find a sphere on 5 vertices in $Z$. On the other hand if $Y$ is on at least 3 triangles then all edges of $S$ belong to $\lk_{Z'}(v)$ and therefore to $\lk_Z(v)$. But after expansion $\lk_{Z'}(v)$ is a flag complex. Therefore $S$ bounds an included tetrahedron of $Z'$. 

Now suppose that $Z$ is a minimal counterexample to our lemma. If $Z$ has a vertex $v$ so that $\overline{\lk_Z(v)}$ is 1-collapsible then we may expand $Z$ to take the flag closure of the link and obtain $Z'$. Now $Z'$ is simple homotopy equivalent to $Z$ and $\lk_{Z'}(v)$ is 1-collapsible. Thus we may 2-collapse $Z'$ around $v$ and obtain a new complex $Z''$ simple homotopy equivalent to $Z'$ where $v$ does not belong to any triangles of $Z''$. Thus $Z''$ is $Z' \setminus \{v\}$ together with $v$ and edges from $v$ to $Z' \setminus \{v\}$. Moreover $Z' \setminus \{v\}$ is essentially 2-sphere free and satisfies the essential density bound. Thus by minimality of $Z$ there is a sequence of collapses and expansions in $Z' \setminus \{v\}$ that transform it to a graph. But then $Z''$ is simple homotopy equivalent to a graph, so $Z$ is too. It follows that $Z$ does not have a vertex $v$ so that $\overline{\lk_Z(v)}$ is 1-collapsible.

We can now conclude that the minimum degree of $Z^{(1)}$ is 4. Moreover if $Z$ has a vertex of degree 4 whose neighbors all have degree 4 too, then $Z$ either contains a free edge so we can perform an elementary collapse and then have a smaller counterexample, or else $Z$ contains a copy of $\diamond_2$ as in the proof of Claim \ref{claim:bigneighbor} violating the essentially 2-sphere free condition. Thus each vertex of $Z$ of degree 4 has a neighbor of degree larger than 4, and just as in the proof of Lemma \ref{simplecores}, we have that $Z$ has density at least 25/12. 
\end{proof}

\begin{proof}[Proof of Theorem \ref{almostaspherical}]
For $p < n^{-\alpha}$ with $\alpha > 12/25$, with high probability $X$ satisfies all of the conditions of Lemma \ref{additionalproperties} with $\ell = L$. Let $Y$ be obtained from $X$ by deleting a triangle from each copy of $\diamond_2$ and 3-collapsing. Clearly $Y$ is 2-dimensional and $\pi_1(Y) = \pi_1(X)$, we prove that $Y$ is aspherical. We first claim that $Y^{(2)}$ cannot contain a 2-sphere on at most $2L$ vertices.

The 1-skeleton of a a triangulated sphere on $v \geq 4$ vertices has density $3 - 6/v$. By condition (6) then $X$ cannot contain a triangulated sphere on more than 6 vertices and at most $2L$ vertices. Thus $Y$ cannot contain triangulated sphere of $v$ vertices for $7 \leq v \leq 2L$, so it remains to check spheres on at most 6 vertices. The only triangulated spheres on at most 6 vertices are $\diamond_2$ and \emph{stacked polytopes}, i.e. spheres that can be obtained from $\partial \Delta_3$ by successively replacing a triangle by the cone over its boundary (commonly called a \emph{bistellar flip}). By construction $Y$ does not contain any copies of $\diamond_2$. 

If $Y$ contains $\partial \Delta_3$ then this means that the interior was removed by collapsing it into a 4-dimensional simplex. Let $\sigma$ be such a tetrahedron with $\sigma \subseteq \tau$ so that $(\sigma, \tau)$ was an elementary collapse. By condition (3) however, we have that $X$ does not contain two 4-simplices that meet at a tetrahedron. Thus after collapsing away all the 4-simplices every tetrahedron of $\tau$ other than $\sigma$ still remains. So we have a triangulated 3-sphere with a single tetrahedron removed. Now when we collapse away triangle--tetrahedron pairs, the only way to get started on collapsing $\tau \setminus \sigma$ is to collapse away one of the triangles of $\sigma$, but then $\partial \sigma$ does not belong to $Y$. 

If $Y$ contains a triangulated sphere on 5 vertices then $X$ has two tetrahedra $\sigma_1$, $\sigma_2$ meeting at a triangle so that after collapse the boundary of the ball with facets $\sigma_1$ and $\sigma_2$ remains. Thus one of $\sigma_1$ or $\sigma_2$ must be collapsed away by a tetrahedron--4-simplex collapse. If not then the only way to get started on collapsing $\sigma_1$ and $\sigma_2$ is to collapse at one of the triangles on the boundary which contradicts the boundary remaining in $Y$. By (3) though, in this case $\sigma_1 \cup \sigma_2$ must form the vertices of a 4-simplex. Without loss of generality then $\sigma_1$ is collapsed with this 4-simplex and the sequence of triangle--tetrahedron collapses restricted to the remaining punctured 3-sphere leaves behind a 2-complex that contains a triangulated sphere. But this is a homotopy equivalence between a contractible space and a 2-complex with homology in degree 2, and that's impossible. 

Lastly suppose that $Y$ contains the boundary of three tetrahedra $\sigma_1, \sigma_2, \sigma_3$ so that $\sigma_1$ and $\sigma_2$ meet at a triangle and $\sigma_2$ and $\sigma_3$ meet at a triangle but $\sigma_1$ and $\sigma_3$ don't meet at a triangle. By (3) then none of the $\sigma_i$ belong to a 4-simplex, so they are all collapsed away by triangle--tetrahedron elementary collapses, and to get started we have to use a triangle on the boundary. So having exhausted all cases we've proved $Y$ does not contain a sphere on at most $2L$ vertices. 

To show that $Y$ is aspherical it suffices by (5) to prove that for any $Z \subseteq Y$ on at most $L$ vertices, $Z$ is aspherical. Suppose that $Z$ is such a complex. By (6) then the graph of $Z$ has essential density less than 25/12. Additionally since $Y$ does not contain a triangulated 2-sphere, $Z$ is essentially 2-sphere free, so by Lemma \ref{Zlemma}, $Z$ is simple homotopy equivalent to a graph, and hence $Z$ is aspherical, and this completes the proof. 
\end{proof}

Finally we prove Lemma \ref{additionalproperties}.
\begin{proof}[Proof of Lemma \ref{additionalproperties}]
We observe that (3) implies (1). Let $G_1$ be the graph of a 5-clique together with a vertex $v$ adjacent exactly to a triangle of the 5-clique. Then $G_1$ has 6 vertices and 13 edges, so the expected number of copies of $G_1$ in $G \sim G(n, p)$ is 
\[O(n^{6 - 13 (12/25)}) = o(1).\]
So by Markov's inequality (3) and therefore (1) occur with high probability. 

For (2), let $G_2$ be the graph obtained from the graph of $\diamond_2$ by adding a new vertex $v$ adjacent exactly to a triangle of $\diamond_2$. Then $G_2$ has 7 vertices and 15 edges so the expected number of such complexes is $O(n^{7 - 15 \alpha})$ which is $o(1)$ since $\alpha > 12/25 > 7/15$. So no copy of $\diamond_2$ has a triangle intersecting a tetrahedron with a vertex outside of that copy of $\diamond_2$. The only other way a copy of $\diamond_2$ can contain a nonmaximal triangle is if it isn't induced. The expected number of noninduced copies of $\diamond_2$ is $O(n^{6 - 13 \alpha}) = o(1)$ since $\alpha > 12/25 > 6/13$.

Condition (4) is by the main result of \cite{Malen} that $\alpha > 1/d$ implies that $X \sim X(n, n^{-\alpha})$ is $d$-collapsible. It is also implied by the $d = 3$ case of Theorem \ref{realmaintheorem}.

Condition (5) is simply a restatement of Lemma \ref{CFH61} of \cite{CostaFarberHorak}.

Condition (6) is by a standard first moment argument. A graph $H$ with $k$ vertices density at least $25/12$ occurs in $G(n, p)$ with probability $O((n^{1 - \alpha (25/12)})^k)$ and as $1 - \alpha(25/12) < 0$ we can sum over all the finitely many graphs on at most $\ell$ vertices with density at least $25/12$ to get a $o(1)$ bound on the expected number of included subgraphs. 
\end{proof}

\section{Concluding remarks}
The natural remaining questions are what are the sharp thresholds, should they exist, for:
\begin{itemize}
\item $X \sim X(n, p)$ to go from almost $d$-collapsible to not almost $d$-collapsible,
\item $X \sim X(n, p)$ to go from having all homology in degree $d$ generated by copies of $\diamond_d$ to having homology not generated by copies of $\diamond_d$, and
\item $X \sim X(n, p)$ to go from $\pi_1(X)$ free to $\pi_1(X)$ nonfree.
\end{itemize}

For the Linial--Meshulam model $Y_d(n, p)$ the ``trivial $d$-cycles" are boundaries of a $(d + 1)$-simplex. These are Poisson distributed and \cite{ALLM}, \cite{AL2} establish the constant $\gamma_d$ so that $\frac{\gamma_d}{n}$ is the sharp threshold for $Y \sim Y_d(n, p)$ to go from almost $d$-collapsible (replacing $\diamond_d$ in the definition here by $\partial \Delta_{d + 1}$) to not almost $d$-collapsible. The threshold in $Y \sim Y_d(n, p)$ to go from all homology generated by copies of $\partial \Delta_{d + 1}$ to some homology coming from other subcomplexes is at $c_d/n$ for an explicit constant $c_d$ is established in \cite{AL}, \cite{LP}. The proofs rely on approximating the local behavior of $Y_d(n, p)$ by a Galton--Watson process and counting $d$-dimensional cores. Based on what's been established, and letting $\gamma_d$ and $c_d$ refer to the critical constants in $Y_d(n, p)$ we make the following conjecture. 

\begin{conjecture}\label{conjecturedthreshold}
For $X \sim X(n, c/\sqrt[d]{n})$ if
\begin{itemize}
\item $0 < c < \sqrt[d]{\gamma_d}$ then with high probability $X$ is almost $d$-collapsible.
\item $\sqrt[d]{\gamma_d} < c < \sqrt[d]{c_d}$ then with high probability $X$ is not almost $d$-collapsible, but $\beta_d(X; \R)$ is generated by embedded copies of $\diamond_d$.
\item $\sqrt[d]{c_d} < c$ then with high probability $\beta_d(X; \kk) = \Theta(n^{(d + 1)/2})$ for any coefficient field $\kk$.
\end{itemize}
\end{conjecture}

The constants $c_d$ and $\gamma_d$ are described fully in \cite{LP}, where the following table of the approximation of the first few values also appears: \\
\begin{center}
\begin{tabular}{c | c | c | c | c }
$d$ & 2 & 3 & 4 & 5 \\ \hline
$\gamma_d$ & 2.455 & 3.089 & 3.509 & 3.822 \\
$c_d$ & 2.754 & 3.907 & 4.962 & 5.984 \\
\end{tabular}
\end{center}
Asymptotically $\gamma_d$ is order $\log(d)$ and $c_d$ is very slightly smaller than $d + 1$. 

Regarding the fundamental group, \cite{FreeGroup} proves the following for $Y \sim Y_2(n, p)$. These are still the best known bounds on a sharp threshold for $\pi_1(Y_2(n, p))$ to be free.

\begin{theorem}\cite{FreeGroup}
If $c < \gamma_2$ and $Y \sim Y_2(n, c/n)$ then with high probability $\pi_1(Y)$ is free, while if $c > c_2$ and $Y \sim Y_2(n, c/n)$ then with high probability $\pi_1(Y)$ is not free.
\end{theorem}
Naturally Conjecture \ref{conjecturedthreshold} together with Theorem \ref{almostaspherical} would imply that $\pi_1(X(n, p))$ is free if $p < \frac{\sqrt{\gamma_2}}{\sqrt{n}} \approx \frac{1.567}{\sqrt{n}}$ and is not free if $p > \frac{\sqrt{c_2}}{\sqrt{n}} \approx \frac{1.660}{\sqrt{n}}$. It seems a conjecture for what the true threshold is between these two bounds would basically be just a guess. It could even be that there is no sharp threshold for $\pi_1(X)$ to go from free to nonfree.

Finally, the sharp threshold for emergence of $d$th homology in $Y_d(n, p)$ is apparently closely tied to the torsion burst in the random model. The torsion burst is a phenomenon observe experimentally in \cite{KLNP}, but as of now no proof exists to explain it. In the experiments right before the first cycle in $d$th homology, other than a copy of the $(d + 1)$-simplex boundary, appears, a huge torsion group appears in the $(d - 1)$st homology group. This torsion group vanishes soon after. A few experiments were conducted in \cite{KLNP} to search for torsion in $X(n, p)$, but none was found. However based on what we've proved here it seems plausible at least that the reason none was observed in experiments was because the experiments were on too few vertices. Perhaps when $n$ is large enough, one would observe a torsion burst in $X(n, p)$. 

As $Y_d(n, p)$ and $X(n, p)$ are special cases of the \emph{multiparameter model} first defined in \cite{CostaFarberMultiparameter}, it would be reasonable to examine one-sided sharp (or two-sided sharp) thresholds for $d$-collapsibility and nonvanishing of $d$th homology in the multiparameter model.

\section*{Acknowledgements}
The author thanks Anton Dochtermann for questions motivating this direction of research and Florian Frick and Matthew Kahle for helpful comments on an earlier draft.
\bibliography{ResearchBibliography}

\providecommand{\MR}[1]{}
\providecommand{\bysame}{\leavevmode\hbox to3em{\hrulefill}\thinspace}
\providecommand{\MR}{\relax\ifhmode\unskip\space\fi MR }
\providecommand{\MRhref}[2]{%
  \href{http://www.ams.org/mathscinet-getitem?mr=#1}{#2}
}
\providecommand{\href}[2]{#2}
\begin{thebibliography}{10}

\bibitem{AL2}
L.~Aronshtam and N.~Linial, \emph{The threshold for {$d$}-collapsibility in
  random complexes}, Random Structures Algorithms \textbf{48} (2016), no.~2,
  260--269. \MR{3449598}

\bibitem{AL}
Lior Aronshtam and Nathan Linial, \emph{When does the top homology of a random
  simplicial complex vanish?}, Random Structures Algorithms \textbf{46} (2015),
  no.~1, 26--35. \MR{3291292}

\bibitem{ALLM}
Lior Aronshtam, Nathan Linial, Tomasz {\L}uczak, and Roy Meshulam,
  \emph{Collapsibility and vanishing of top homology in random simplicial
  complexes}, Discrete Comput. Geom. \textbf{49} (2013), no.~2, 317--334.
  \MR{3017914}

\bibitem{benedetti2009dunce}
Bruno Benedetti and Frank~H Lutz, \emph{The dunce hat in a minimal
  non-extendably collapsible $3$-ball}, Electronic Geometry Model No.
  2013.10.001 (2013).

\bibitem{BollobasDensity}
B\'{e}la Bollob\'{a}s, \emph{Threshold functions for small subgraphs}, Math.
  Proc. Cambridge Philos. Soc. \textbf{90} (1981), no.~2, 197--206. \MR{620729}

\bibitem{CF}
A.~E. Costa and M.~Farber, \emph{The asphericity of random 2-dimensional
  complexes}, Random Structures Algorithms \textbf{46} (2015), no.~2, 261--273.
  \MR{3302897}

\bibitem{CostaFarberMultiparameter}
Armindo Costa and Michael Farber, \emph{Random simplicial complexes},
  Configuration spaces, Springer INdAM Ser., vol.~14, Springer, [Cham], 2016,
  pp.~129--153. \MR{3615731}

\bibitem{CostaFarberHorak}
Armindo Costa, Michael Farber, and Danijela Horak, \emph{Fundamental groups of
  clique complexes of random graphs}, Trans. London Math. Soc. \textbf{2}
  (2015), no.~1, 1--32. \MR{3355576}

\bibitem{DochtermannNewman}
Anton Dochtermann and Andrew Newman, \emph{Random subcomplexes and {B}etti
  numbers of random edge ideals}, arXiv: 2104.12882. To appear in \emph{IMRN}.

\bibitem{ER}
P.~Erd\H{o}s and A.~R\'{e}nyi, \emph{On the evolution of random graphs}, Magyar
  Tud. Akad. Mat. Kutat\'{o} Int. K\"{o}zl. \textbf{5} (1960), 17--61.
  \MR{125031}

\bibitem{ErmanYang}
Daniel Erman and Jay Yang, \emph{Random flag complexes and asymptotic
  syzygies}, Algebra Number Theory \textbf{12} (2018), no.~9, 2151--2166.
  \MR{3894431}

\bibitem{FriezeRandomGraphs}
Alan Frieze and Micha\l{} Karo\'{n}ski, \emph{Introduction to random graphs},
  Cambridge University Press, Cambridge, 2016. \MR{3675279}

\bibitem{KahleRandomClique}
Matthew Kahle, \emph{Topology of random clique complexes}, Discrete Math.
  \textbf{309} (2009), no.~6, 1658--1671. \MR{2510573}

\bibitem{KahleRational}
\bysame, \emph{Sharp vanishing thresholds for cohomology of random flag
  complexes}, Ann. of Math. (2) \textbf{179} (2014), no.~3, 1085--1107.
  \MR{3171759}

\bibitem{KLNP}
Matthew Kahle, Frank~H. Lutz, Andrew Newman, and Kyle Parsons,
  \emph{Cohen-{L}enstra heuristics for torsion in homology of random
  complexes}, Exp. Math. \textbf{29} (2020), no.~3, 347--359. \MR{4134833}

\bibitem{LP}
Nathan Linial and Yuval Peled, \emph{On the phase transition in random
  simplicial complexes}, Ann. of Math. (2) \textbf{184} (2016), no.~3,
  745--773. \MR{3549622}

\bibitem{Malen}
Greg Malen, \emph{Collapsibility of random clique complexes}, arXiv:
  1903.05055.

\bibitem{FreeGroup}
Andrew Newman, \emph{Freeness of the random fundamental group}, J. Topol. Anal.
  \textbf{12} (2020), no.~1, 29--35. \MR{4080093}

\end{thebibliography}
\bibliographystyle{amsplain}

\end{document}